\documentclass[a4paper]{article}
\usepackage{mathtext}
\usepackage{amssymb}
\usepackage{amsfonts}
\usepackage{amsthm}
\usepackage{amsmath}
\usepackage{latexsym}
\usepackage{geometry}
\usepackage{graphicx}
\usepackage{mathdots}
\usepackage{amsmath,amscd}
\usepackage{xypic}
\usepackage{color}

\newtheorem{theorem}{Theorem}

\newtheorem{definition}[theorem]{Definition}

\newtheorem{lemma}[theorem]{Lemma}

\newtheorem{remark}[theorem]{Remark}

\input{tcilatex}
\begin{document}

\title{On the Classification of $G$-Graded Twisted Algebras over Finite
Abelian Groups}
\author{Juan P. Hernandez, Juan D. Velez, Luis A. Wills-Toro, Edisson
Gallego.}
\date{}
\maketitle

\begin{abstract}
Let $G$ be a group and let $W$ be an algebra over a field $K$. We will say
that $W$ is a $G$-graded twisted algebra if $W$ can be written as $W=\oplus
_{g\in G}W_{g}$ with $W_{a}W_{b}\subset W_{ab}$, and where each $W_{g}$ is a
one dimensional $K$-vector space. It is also assumed that $W $ has no
monomial which is a zero divisor which means that for each pair of nonzero
elements $w_{a}\in W_{a}$, $w_{b}\in W_{b}$, $w_{a}\cdot w_{b}\neq 0$. We
also demand that $W$ has a multiplicative identity element. We focus in the
case where $G$ is a finite abelian group and $K=\mathbb{C}$ or $K=\mathbb{R}$%
.\newline
In this article, using methods of group cohomology, we classify all
associative $G$-graded twisted algebras in the case $G$ is a finite abelian
group. On the other hand, by generalizing some of the arguments developed in 
\cite{jwn} we present a classification of all $G$-graded twisted algebras
that satisfy certain symmetry condition.
\end{abstract}

\section{Introduction}

$G$-graded twisted algebras were introduced in \cite{tga}, and independently
in \cite{wt}, as distinguished mathematical structures which arise naturally
in theoretical physics \cite{w1}, \cite{w2}, \cite{w3}, \cite{w4}, \cite{w5}
and \cite{tga2}. A $G$-graded twisted algebra $W$ is an algebra over a
commutative ring $R$ with a $G$-grading, i.e., $W=\oplus _{g\in G}W_{g}$,
with $W_{a}W_{b}\subset W_{ab}$. Each $W_{g}$ is assumed to be a free $R$%
-module of rank one, and we demand that $W$ is free of zero monomial
divisors, i.e., $w_{a}\cdot w_{b}\neq 0$ for every non-zero elements $%
w_{a}\in W_{a} $, $w_{b}\in W_{b}$. We also demand that $W$ has an identity
element $1\in W_{e}$, where $W_{e}$ denotes the graded component
corresponding to the identity element $e\in G$.

A classification of all associative $G$-graded twisted algebras over the
real and complex numbers over finite cyclic groups is achieved in \cite{jwn}
by using standard techniques from group cohomology. In the first part of
this article we generalize the methods developed in \cite{jwn} in order to
provide a classification of all associative $G$-graded twisted algebras over
any finite abelian group. The exact number of distinct (graded) isomorphism
classes is given in Theorem \ref{tm} where a formula to count them is
presented.

The second part of this article deals with the classification problem for
non-associative $G$-graded twisted algebras that satisfy a particular type
of symmetry condition. In Theorem \ref{t3} we also provide an exact formula
to count (up to graded isomorphisms) all symmetric algebras that are graded
over an abelian group. This generalizes the main result obtained in \cite%
{jwn} for cyclic groups.

\section{Definitions and basic notions}

\subsection{$G$-graded twisted algebras over fields.}

\begin{definition}
Let $G$ denote a group. A $G$-graded twisted algebra $W$ is an algebra over
a commutative ring $R$ with a $G$-grading, i.e., $W=\oplus _{g\in G}W_{g}$,
with $W_{a}W_{b}\subset W_{ab}$. Each $W_{g}$ is assumed to be a free $R$%
-module of rank one, and we demand that $W$ is free of zero monomial
divisors, i.e., $w_{a}\cdot w_{b}\neq 0$ for every non-zero elements $%
w_{a}\in W_{a}$, $w_{b}\in W_{b}$. We also demand that $W$ has an identity
element $1\in W_{e}$, where $W_{e}$ denotes the graded component
corresponding to the identity element $e\in G$. In this article $R$\emph{\
always will be a field.}
\end{definition}

As each graded component $W_{g}$ is a vector space of dimension one, each
choice of a non zero element $w_{g}\in W_{g}$, for each $g\in G$, produces a
graded basis $\mathcal{B}=\{w_{g}:g\in G\}$ for $W$. For each such basis
there is a structure constant associated to it, $C_{\mathcal{B}}:G\times
G\rightarrow A$, defined by the identity: $w_{a}\cdot w_{b}=C_{\mathcal{B}%
}(a,b)w_{ab}$. Here $A\subset K^{\ast }$ must be a subgroup of the
multiplicative group of all nonzero elements of $K$, since $W$ has no zero
divisor monomials. From now on we will omit the subscript $\mathcal{B}$ if a
particular basis is clear in the context.

Given a structure constant $C:G\times G\rightarrow A$, we define two
functions $q:G\times G\rightarrow A$ and $r:G\times G\times G\rightarrow A$
as: 
\begin{equation}
\begin{aligned} q(a,b)& =C(a,b)C(b,a)^{-1}\\ r(a,b,c) & =
C(b,c)C(ab,c)^{-1}C(a,bc)C(a,b)^{-1} \end{aligned}  \label{e1}
\end{equation}%
The associativity of elements of $W$ can be described in terms of the
function $r:G\times G\times G\rightarrow A$ as follows: $w_{a}\cdot
(w_{b}\cdot w_{c})=r(a,b,c)(w_{a}\cdot w_{b})\cdot w_{c}$. When $G$ is an
abelian group, the commutativity of elements of $W$ is given in terms of the
function $q:G\times G\rightarrow A$ as $w_{a}\cdot w_{b}=q(a,b)w_{b}\cdot
w_{a}$.

\begin{definition}
A morphism between two $G$-graded twisted $K$-algebras $W=\oplus_{g \in G}
W_{g}$ and $V=\oplus_{g \in G} V_{g}$ is an unitarian homomorphism of $K$%
-algebras $\varphi: W \rightarrow V$. If the homomorphism preserves the
grading, i.e., $\varphi(W_{g}) \subset V_{g}$, we say the morphism is 
\textit{graded}.
\end{definition}

\subsection{Group cohomology.}

In this section we recall some basic definitions and results from the
cohomology of groups that will be needed in the next section. We adhere to
the terminology used in \cite{kb}.

Let $G$ be a group. A $G$-module $M$ is an abelian group with an action
which is compatible with the abelian group structure of $M$. Every abelian
group $M$ may be regarded as a $G$-module with the trivial action. Given a
group $G$ and $A,$ an abelian group, the cohomology of $G$ with coefficients
in $A$, denoted by $H^{\bullet }(G,A)$, may be defined as $\text{Ext}_{%
\mathbb{Z}[G]}^{\bullet }(\mathbb{Z},A)$. A very useful way to compute $%
H^{\bullet }(G,A)$ is the following: Define $C^{n}(G,A)$ as the set of
functions from $G^{n}=$ the product of $n$ copies of $G$, to $A$, and
consider $\partial ^{n}:C^{n}(G,A)\rightarrow C^{n+1}(G,A)$, the boundary
map given by the formula: 
\begin{equation*}
\partial ^{n}(f)(g_{1},\dots ,g_{n+1})=g_{1}f(g_{2},\dots
,g_{n+1})+\sum_{i=1}^{n}(-1)^{i}f(g_{1},\dots ,g_{i-1},g_{i}g_{i+1},\dots
,g_{n+1})+f(g_{1},\dots ,g_{n}).
\end{equation*}%
Then, the cohomology $H^{\bullet }(G,A)$ turns out to be the homology of the
complex: 
\begin{equation*}
\xymatrix{\cdots \ar[r] & C^{n-1}(G,A) \ar[r]^{\partial^{n-1}} & C^{n}(G,A)
\ar[r]^{\partial^{n}} & C^{n+1}(G,A) \ar[r] & \cdots },
\end{equation*}%
(see \cite{kb}, page 19).

\section{Computation of $H^{2}(G,A)$ when $G$ is a finite abelian group.}

When $G$ is a finite cyclic group there is a particular free resolution for $%
\mathbb{Z}$, regarded as a $\mathbb{Z}[G]$-modules with the trivial action.
If we denote $\mathbb{Z}[t]$ as the quotient $\frac{\mathbb{Z}[T]}{(T^{n}-1)}
$, and let $N $ be the sum $1+t+t^{2}+\cdots +t^{n-1}$, then clearly $%
\mathbb{Z}\cong \frac{\mathbb{Z}[t]}{(t-1)},$ and the following complex is
exact: 
\begin{equation}
\xymatrix{\cdots \ar[r]^{N} & \mathbb{Z}[t] \ar[r]^{t-1} & \mathbb{Z}[t]
\ar[r]^{N} & \mathbb{Z}[t] \ar[r]^{t-1} & \mathbb{Z}[t] \ar[r]^{\pi} &
\mathbb{Z} \ar[r] & 0}
\end{equation}%
By suppressing $\mathbb{Z},$ and then applying the contravariant functor $%
\text{Hom}_{\mathbb{Z}[t]}(-,A),$ where $A$ is regarded as a $G$-module with
the trivial action, we obtain the following commutative diagram: 
\begin{equation}
\xymatrix{0 \ar[r] & \textrm{Hom}_{\mathbb{Z}[t]}(\mathbb{Z}[t], A)
\ar[r]^{t-1} \ar[d]_{\lambda_{0}} &
\textrm{Hom}_{\mathbb{Z}[t]}(\mathbb{Z}[t], A) \ar[r]^{N}
\ar[d]_{\lambda_{1}} & \textrm{Hom}_{\mathbb{Z}[t]}(\mathbb{Z}[t], A)
\ar[r]^{t-1} \ar[d]_{\lambda_{2}} &
\textrm{Hom}_{\mathbb{Z}[t]}(\mathbb{Z}[t], A) \ar[r] \ar[d]_{\lambda_{3}} &
\cdots \\ 0 \ar[r] & A \ar[r]^{D^{0}} & A \ar[r]^{D^{1}} & A \ar[r]^{D^{2}}
& A \ar[r] & \cdots}  \label{e7}
\end{equation}%
The map $\lambda _{i}:\text{Hom}_{\mathbb{Z}[t]}(\mathbb{Z}[t],A)\rightarrow
A$ is the isomorphism given by $\lambda _{i}(\varphi )=\varphi (1)$, with
inverse $\lambda _{i}^{-1}:A\rightarrow \text{Hom}_{\mathbb{Z}[t]}(\mathbb{Z}%
[t],A)$ given by $\lambda _{i}^{-1}(a)=\varphi _{a}$, where $\varphi
_{a}(1)=a$ and $D^{j}=\lambda _{j+1}\circ \delta _{j}\circ \lambda _{j}^{-1}$%
. Here, $\delta _{0}=t-1,$ $\delta _{1}=N,$ $\delta _{2}=t-1,$ and $\delta
_{3}=N$. Thus, 
\begin{equation*}
D^{1}(a)=N\varphi _{a}(1)=N\cdot a=na,\text{ and }D^{2}(a)=(t-1)\varphi
_{a}(1)=(t-1)\cdot a=0.
\end{equation*}%
Therefore, by taking the second homology we obtain $H^{2}(G,A)\cong $ ker$%
(D^{2})/$im$(D^{1})=A/nA.$

We will need the following lemma:

\begin{lemma}
If $\xymatrix{F_{\bullet} \ar[r]^{\delta_{\bullet}} & \mathbb{Z}
\ar[r] & 0}$ is a free resolution of $\mathbb{Z}[G_{1}]$-modules and $%
\xymatrix{F'_{\bullet} \ar[r]^{\delta'_{\bullet}} &
\mathbb{Z} \ar[r] & 0}$ is a free resolution of $\mathbb{Z}[G_{2}]$-modules,
then the tensor product of these two resolutions $H_{\bullet }\overset{%
D_{_{\bullet }}}{\rightarrow }\mathbb{Z}\rightarrow 0$ is a free resolution
of $\mathbb{Z}[G_{1}\times G_{2}]$-modules, where $H_{n}=%
\bigoplus_{i=0}^{n}F_{n-i}\otimes F_{i}^{\prime }$, and for each $a\otimes
b\in F_{n-i}\otimes F_{i}^{\prime },$ the boundary homomorphisms are defined
as 
\begin{equation*}
D_{n}(a\otimes b)=\delta _{n-i}(a)\otimes b+(-1)^{n-i}a\otimes \delta
_{i}^{\prime }(b).
\end{equation*}
(See \cite{kb}, page 107.)
\end{lemma}

Now we are ready to prove the main result of this section.

\begin{theorem}
\label{tm} Let $G$ be a finite abelian group, presented as $G=G_{1}\times
\cdots \times G_{k}$, where $G_{i}$ is a cyclic group of order $n_{i}$. Let $%
A$ be any abelian group. Then 
\begin{equation*}
H^{2}(G,A)\cong \left( \bigoplus_{i=0}^{k}\frac{A}{n_{i}A}\right) \bigoplus
\left( \bigoplus_{1\leq i<j\leq k}\text{Ann}_{A}(n_{i})\cap \text{Ann}%
_{A}(n_{j})\right) .
\end{equation*}
\end{theorem}

\begin{proof}
By the lemma above, the tensor product of the free resolutions of $\mathbb{Z}
$, regarded as a $\mathbb{Z}[G_{i}]$-module, gives a free resolution of $%
\mathbb{Z}$ as a $\mathbb{Z}[G]$-module. On the other hand, $\mathbb{Z}%
[G]\cong \mathbb{Z}[G_{1}]\times \cdots \times \mathbb{Z}[G_{k}]$. This
product in turn is isomorphic to $\mathbb{Z}[T_{1},\dots
,T_{k}]/(T^{n_{1}}-1,\dots ,T^{n_{k}}-1)=\mathbb{Z}[t_{1},\dots ,t_{k}]$,
where $t_{i}$ denotes the class of $T_{i}$.\

We use induction on $k$: The previous lemma gives the following free
resolution of $\mathbb{Z}$ as a $\mathbb{Z}[G_{1}\times \cdots \times
G_{k-1}]$-module:
\begin{equation}
\xymatrix{& F_{3} \ar@{=}[d] & F_{2} \ar@{=}[d] & F_{1} \ar@{=}[d] & F_{0}
\ar@{=}[d] & &\\ \ar[r] & \mathbb{Z}[t_{1}, \dots, t_{k-1}]^{a_{k-1}}
\ar[r]^{\Delta_{2}^{k-1}} & \mathbb{Z}[t_{1}, \dots, t_{k-1}]^{b_{k-1}}
\ar[r]^{\Delta_{1}^{k-1}} & \mathbb{Z}[t_{1}, \dots, t_{k-1}]^{c_{k-1}}
\ar[r]^{\Delta_{0}^{k-1}} & \mathbb{Z}[t_{1}, \dots, t_{k-1}] \ar[r] &
\mathbb{Z} \ar[r] & 0}  \label{e18}
\end{equation}%
By tensoring with the following free resolution of $\mathbb{Z}$ as a $%
\mathbb{Z}[G_{k}]$-module
\begin{equation}
\xymatrix{& F'_{3} \ar@{=}[d] & F'_{2} \ar@{=}[d] & F'_{1} \ar@{=}[d] &
F'_{0} \ar@{=}[d] & &\\ \ar[r] & \mathbb{Z}[t_{k}] \ar[r]^{\delta_{2}^{k}} &
\mathbb{Z}[t_{k}] \ar[r]^{\delta_{1}^{k}} & \mathbb{Z}[t_{k}]
\ar[r]^{\delta_{0}^{k}} & \mathbb{Z}[t_{k}] \ar[r] & \mathbb{Z} \ar[r] & 0,}
\label{e19}
\end{equation}%
where $\delta _{0}^{k}=t_{k}-1$, and $\delta _{1}^{k}=N_{k}=1+t_{k}+\cdots
+t_{k}^{n_{k}-1}$, we obtain a free resolution of $\mathbb{Z}$ as a $\mathbb{%
Z}[G_{1}\times \cdots \times G_{k}]$-module:
\begin{equation}
\xymatrix{ \ar[r] & \mathbb{Z}[t_{1}, \dots, t_{k}]^{a_{k}}
\ar[r]^{\Delta_{2}^{k}} & \mathbb{Z}[t_{1}, \dots, t_{k}]^{b_{k}}
\ar[r]^{\Delta_{1}^{k}} & \mathbb{Z}[t_{1}, \dots, t_{k}]^{c_{k}}
\ar[r]^{\Delta_{0}^{k}} & \mathbb{Z}[t_{1}, \dots, t_{k}] \ar[r] &
\mathbb{Z} \ar[r] & 0.}  \label{e20}
\end{equation}%
Here,
\begin{equation*}
\mathbb{Z}[t_{1},\dots ,t_{k}]^{a_{k}}=\bigoplus_{i=0}^{3}F_{3-i}\otimes
F_{i}^{\prime }=\mathbb{Z}[t_{1},\dots ,t_{k}]^{a_{k-1}}\bigoplus \mathbb{Z}%
[t_{1},\dots ,t_{k}]^{b_{k-1}}\bigoplus \mathbb{Z}[t_{1},\dots
,t_{k}]^{c_{k-1}}\bigoplus \mathbb{Z}[t_{1},\dots ,t_{k}],
\end{equation*}%
\begin{equation*}
\mathbb{Z}[t_{1},\dots ,t_{k}]^{b_{k}}=\bigoplus_{i=0}^{2}F_{2-i}\otimes
F_{i}^{\prime }=\mathbb{Z}[t_{1},\dots ,t_{k}]^{b_{k-1}}\bigoplus \mathbb{Z}%
[t_{1},\dots ,t_{k}]^{c_{k-1}}\bigoplus \mathbb{Z}[t_{1},\dots ,t_{k}],
\end{equation*}%
and
\begin{equation*}
\mathbb{Z}[t_{1},\dots ,t_{k}]^{c_{k}}=\bigoplus_{i=0}^{1}F_{1-i}\otimes
F_{i}^{\prime }=\mathbb{Z}[t_{1},\dots ,t_{k}]^{c_{k-1}}\bigoplus \mathbb{Z}%
[t_{1},\dots ,t_{k}].
\end{equation*}%
Therefore,
\begin{eqnarray*}
a_{k} &=&a_{k-1}+b_{k-1}+c_{k-1}+1=a_{k-1}+\binom{k}{2}+(k-1)+1, \\
b_{k} &=&b_{k-1}+c_{k-1}+1=\binom{k}{2}+(k-1)+1=\binom{k+1}{2},\text{ and} \\
c_{k} &=&c_{k-1}+1=(k-1)+1=k.
\end{eqnarray*}%
Furthermore, we also get recursive equations for the boundary maps $\Delta
_{0}^{k},\Delta _{1}^{k}$ and $\Delta _{2}^{k}$:
\begin{equation*}
\Delta _{0}^{k}((p_{1},\dots ,p_{k-1}),q)=\Delta _{0}^{k-1}(p_{1},\dots
,p_{k-1})+\delta _{0}^{k}(q),
\end{equation*}%
\begin{eqnarray*}
&&\Delta _{1}^{k}((p_{1},\dots ,p_{\binom{k}{2}}),(q_{1},\dots ,q_{k-1}),r)
\\
&=&(\Delta _{1}^{k-1}(p_{1},\dots ,p_{\binom{k}{2}})-\delta
_{0}^{k}(q_{1},\dots ,q_{k-1}),\Delta _{0}^{k-1}(q_{1},\dots
,q_{k-1})+\delta _{1}^{k}(r)),
\end{eqnarray*}%
\begin{equation}
\begin{aligned}& \Delta_{2}^{k}((p_{1}, \dots , p_{a_{k-1}}),(q_{1}, \dots,
q_{\binom{k}{2}}),(r_{1}, \dots, r_{k-1}),l) =\\ & =(\Delta_{2}^{k-1}(p_{1},
\dots, p_{a_{k-1}})+\delta_{0}^{k}(q_{1}, \dots,
q_{\binom{k}{2}}),\Delta_{1}^{k-1}(q_{1}, \dots,
q_{\binom{k}{2}})-\delta_{1}^{k}(r_{1}, \dots, r_{k-1}),
\Delta_{0}^{k-1}(r_{1}, \dots, r_{k-1})+\delta_{2}^{k}(l))\end{aligned}
\label{e21}
\end{equation}%
where $\delta _{i}^{k}(a_{1},\dots ,a_{n})=(\delta _{i}^{k}(a_{1}),\dots
,\delta _{i}^{k}(a_{n}))$. Suppressing $\mathbb{Z}$, and then applying the
functor $\text{Hom}_{\mathbb{Z}[G]}(\_,A)$ in (\ref{e20}), we get a complex
which is isomorphic to the complex:
\begin{equation*}
\xymatrix{0 \ar[r] & A \ar[r]^{D^{0}_{k}} & A^{k-1} \bigoplus A
\ar[r]^{D^{1}_{k}} & A^{\binom{k}{2}} \bigoplus A^{k-1} \bigoplus A
\ar[r]^{D^{2}_{k}} & A^{a_{k-1}} \bigoplus A^{\binom{k}{2}} \bigoplus
A^{k-1} \bigoplus A \ar[r] & \cdots},
\end{equation*}%
where the boundary maps $D_{k}^{1}$ and $D_{k}^{2}$ are also given
recursively by:
\begin{equation*}
\begin{aligned} & D^{1}_{k}((x_{1}, \dots, x_{k-1}),y) = (D^{1}_{k-1}(x_{1},
\dots, x_{k-1}),0,0, \dots ,0, n_{k} \cdot y) \in A^{\binom{k}{2}} \oplus
A^{k-1} \oplus A\\ & =
(n_{1}x_{1},0,n_{2}x_{2},0,0,n_{3}x_{3},0,0,0,n_{4}x_{4},0,0,0,0,n_{5}x_{5},0,0,0,0,0,n_{6}x_{6}, \dots, n_{k-1}x_{k-1},0,0,\dots, 0, n_{k}y). \end{aligned}%
\end{equation*}%
And by
\begin{equation*}
\begin{aligned} & D^{2}_{k}((x_{1}, \dots, x_{b_{k-1}}),(y_{1}, \dots,
y_{k-1}),z) = \\ & = (D^{2}_{k-1}(x_{1}, \dots,
x_{b_{k-1}}),D^{1}_{k-1}(y_{1}, \dots, y_{k-1}),-n_{k}y_{1}, \dots,
-n_{k}y_{k-1},0) \in A^{a_{k-1}} \oplus A^{\binom{k}{2}} \oplus A^{k-1}
\oplus A\\ & =(D_{k-1}^{2}(x_{1}, \dots,
x_{b_{k-1}}),n_{1}y_{1},0,n_{2}y_{2},0,0,n_{3}y_{3},0,0,0,n_{4}y_{4},0,0,0,0,n_{5}y_{5},0,0,0,0,0,n_{6}y_{6}, \dots,\\ & \dots, n_{k-2}y_{k-2},0,0,\dots, 0, n_{k-1}y_{k-1}, -n_{k}y_{1}, \dots, -n_{k}y_{k-1},0). \end{aligned}
\end{equation*}%
Therefore, $\mathrm{im}(D_{k}^{1})=\mathrm{im}(D_{k-1}^{1})\oplus 0\oplus
0\oplus \cdots \oplus 0\oplus n_{k}A,$ and
\begin{equation*}
\text{ker}(D_{k}^{2})=\text{ker}(D_{k-1}^{2})\oplus \text{Ann}%
_{A}(n_{1})\cap \text{Ann}_{A}(n_{k})\oplus \cdots \oplus \text{Ann}%
_{A}(n_{k-1})\cap \text{Ann}_{A}(n_{k})\oplus A.
\end{equation*}%
Hence, we obtain the following recursive formula:
\begin{equation*}
H^{2}(G,A)\cong \frac{\text{ker}(D_{k-1}^{2})}{\text{im}(D_{k-1}^{1})}\oplus
\text{Ann}_{A}(n_{1})\cap \text{Ann}_{A}(n_{k})\oplus \cdots \oplus \text{Ann%
}_{A}(n_{k-1})\cap \text{Ann}_{A}(n_{k})\oplus \frac{A}{n_{k}A}.
\end{equation*}%
By induction we know that
\begin{equation*}
H^{2}(L,A)\cong \left( \bigoplus_{i=1}^{k-1}\frac{A}{n_{i}A}\right)
\bigoplus \left( \bigoplus_{1\leq i<j\leq k-1}\text{Ann}_{A}(n_{i})\cap
\text{Ann}_{A}(n_{j})\right) ,
\end{equation*}%
where $L=G_{1}\times \cdots \times G_{k-1}$. Since ker$(D_{k-1}^{2})/$im$%
(D_{k-1}^{1})=H^{2}(L,A)$, it follows that
\begin{equation*}
H^{2}(G,A)\cong \left( \bigoplus_{i=1}^{k}\frac{A}{n_{i}A}\right) \bigoplus
\left( \bigoplus_{1\leq i<j\leq k}\text{Ann}_{A}(n_{i})\cap \text{Ann}%
_{A}(n_{j})\right) .
\end{equation*}
\end{proof}

\section{Classification of associative $G$-graded twisted $K$-algebras, when 
$G$ is a finite abelian group and $K=\mathbb{C}$ or $K=\mathbb{R}.$}

We start by recalling the following theorem that was proved in \cite{jwn}.

\begin{theorem}
\label{ti} Let $W=\oplus _{g\in G}W_{g}$ and $V=\oplus _{g\in G}V_{g}$ be $G$%
-graded twisted $K$-algebras with fixed bases $\mathcal{B}$ and $\mathcal{B}%
^{\prime }$ respectively, and let $C_{1},C_{2}$ the corresponding structure
constants. Then $W$ is graded-isomorphic to $V$ if and only if the function $%
C_{1}C_{2}^{-1}$ is in the kernel of $d^{2}:C^{2}(G,K^{\ast })\rightarrow
C^{3}(G,K^{\ast })$ and the class $[C_{1}C_{2}^{-1}]$ is trivial in $%
H^{2}(G,K^{\ast })$ (\cite{jwn}, Page 4).
\end{theorem}

The above theorem implies that two associative $G$-graded twisted algebras $%
W_{1}$ and $W_{2}$ are isomorphic as graded algebras if and only if $%
[C_{1}]=[C_{2}]$ in $H^{2}(G,K^{\ast })$, where $C_{1},C_{2}:G\times
G\rightarrow K^{\ast }$ denote the structure constants for $W_{1}$ and $%
W_{2} $, respectively, and where $K^{\ast }$ is viewed as a trivial $G$%
-module. Hence, in this case the number of non-isomorphic associative $G$%
-graded twisted algebras is given by the cardinality of $H^{2}(G,K^{\ast })$.

Assume $G$ is a finite abelian group presented as $G\cong Z_{n_{1}}\times
Z_{n_{2}}\times \cdots \times Z_{n_{k}}$. First, let us consider the case
where $K=\mathbb{C}$.

As we showed before, 
\begin{equation*}
H^{2}(\mathbb{Z}_{n_{1}}\times \cdots \times \mathbb{Z}_{n_{k}},\mathbb{C}%
^{\ast })\cong \left( \bigoplus_{i=1}^{k}\frac{\mathbb{C}^{\ast }}{n_{i}%
\mathbb{C}^{\ast }}\right) \bigoplus \left( \bigoplus_{1\leq i<j\leq k}\text{%
Ann}_{\mathbb{C}^{\ast }}(n_{i})\cap \text{Ann}_{\mathbb{C}^{\ast
}}(n_{j})\right) .
\end{equation*}%
Note that $\mathbb{C}^{\ast }=n\mathbb{C}^{\ast },$ for all $n\in \mathbb{N}$%
. If $d_{i,j}=\text{gcd}(n_{i},n_{j})$ denotes the greatest common divisor
for $i,j=1,2,\dots ,k$, then $\text{Ann}_{\mathbb{C}^{\ast }}(n_{i})\cap 
\text{Ann}_{\mathbb{C}^{\ast }}(n_{j})=\text{Ann}_{\mathbb{C}^{\ast
}}(d_{i,j})$. Thus, we have the following isomorphism: 
\begin{equation*}
H^{2}(\mathbb{Z}_{n_{1}}\times \cdots \times \mathbb{Z}_{n_{k}},\mathbb{C}%
^{\ast })\cong \bigoplus_{1\leq i<j\leq k}\text{Ann}_{\mathbb{C}^{\ast
}}(d_{i,j})\cong \bigoplus_{1\leq i<j\leq k}\mathbb{Z}_{d_{i,j}}.
\end{equation*}%
Therefore, there are $d$ non-isomorphic associative $G$-graded twisted $%
\mathbb{C}$- algebras, where $d=\prod_{1\leq i<j\leq k}d_{i,j}$.

Now we deal with the case $K=\mathbb{R}$. We know that 
\begin{equation*}
H^{2}(\mathbb{Z}_{n_{1}}\times \cdots \times \mathbb{Z}_{n_{k}},\mathbb{R}%
^{\ast })\cong \left( \bigoplus_{i=1}^{k}\frac{\mathbb{R}^{\ast }}{n_{i}%
\mathbb{R}^{\ast }}\right) \bigoplus \left( \bigoplus_{1\leq i<j\leq k}\text{%
Ann}_{\mathbb{R}^{\ast }}(n_{i})\cap \text{Ann}_{\mathbb{R}^{\ast
}}(n_{j})\right) .
\end{equation*}%
By reorganizing the $n_{j}$'s we may assume, without loss of generality,
that $n_{1},\dots ,n_{s}$ are even and that $n_{s+1},\dots ,n_{k}$ are odd.
We know that $\mathbb{R}^{\ast }/n\mathbb{R}^{\ast }=\{1\},$ if $n$ is odd,
and $\mathbb{R}^{\ast }/n\mathbb{R}^{\ast }=\{1,-1\},$ if $n$ is even, and
that Ann$_{\mathbb{R}^{\ast }}(n)=\{1\}$, if $n$ is odd, and Ann$_{\mathbb{R}%
^{\ast }}(n)=\{1,-1\}$, if $n$ is even. Therefore, 
\begin{equation*}
H^{2}(\mathbb{Z}_{n_{1}}\times \cdots \times \mathbb{Z}_{n_{s}}\times 
\mathbb{Z}_{n_{s+1}}\times \cdots \times \mathbb{Z}_{n_{k}},\mathbb{R}^{\ast
})\cong \left( \bigoplus_{i=1}^{s}\mathbb{Z}_{2}\right) \bigoplus \left(
\bigoplus_{i=1}^{\frac{(s-1)s}{2}}\mathbb{Z}_{2}\right) \cong
\bigoplus_{i=1}^{\frac{s(s+1)}{2}}\mathbb{Z}_{2}.
\end{equation*}%
This formula readily implies the following theorem.

\begin{theorem}
Let $G$ be written as a product $G=\mathbb{Z}_{n_{1}}\times \cdots \times 
\mathbb{Z}_{n_{s}}\times \mathbb{Z}_{n_{s+1}}\times \cdots \times \mathbb{Z}%
_{n_{k}}$ where $n_{1},\dots ,n_{s}$ are even and $n_{s+1},\dots ,n_{k}$ are
odd. Then there are $2^{s(s+1)/2}$ non-isomorphic associative $G$-graded
twisted $\mathbb{R}$-algebras.
\end{theorem}

In the next section we deal with the classification problem in the
non-associative case. We study those algebras satisfying a type of symmetry
that will be called $(1,2)$-symmetry.

\section{Classification of $(1,2)$-symmetric $G$-graded twisted $\mathbb{C}$%
-algebras.}

We recall that given $W$ a $G$-graded twisted $K$-algebra with a fixed basis 
$\mathcal{B}$ and structure constant $C:G\times G\rightarrow K^{\ast }$, the
associative function $r:G\times G\times G\rightarrow K^{\ast }$ was defined
as $r(a,b,c)=C(b,c)C(ab,c)^{-1}C(a,bc)C(a,b)^{-1}$.

\begin{remark}
The function $r:G\times G\times G\rightarrow K^{\ast }$ does not depend on
the choice basis of $W$ (see \cite{jwn}).
\end{remark}

\begin{definition}
Let $W$ be a $G$-graded twisted $K$-algebra. We say that $W$ is $(1,2)$%
-symmetric if $r(a,b,c)=r(b,a,c)$ for every $a,b,c\in G$.
\end{definition}

In \cite{jwn}, it was proved that for $G\cong \mathbb{Z}_{n}$, a cyclic
group, the number of non-(graded) isomorphic $(1,2)$-symmetric $G$-graded
twisted $\mathbb{C}$-algebras with structure constants taking values in a
finite subgroup $A\subset \mathbb{C}^{\ast }$ is given by $%
|R_{n}|^{|G|-2}=|R_{n}|^{n-2}$, where $R_{n}$ denotes the set of $n$-th
roots of unity in $A$. In this section, we provide a generalization of the
arguments used in \cite{jwn} that will allow us to state an equivalent
result for the case of any finite abelian group. For the sake of clarity we
will mainly focus on groups that are the product of only two cyclic groups.
At the end of this section we deal with finite abelian groups in general.
Since the arguments are almost identical to the case of groups that are the
product of two factors, we will limit ourselves to sketch the main arguments.

Suppose $G$ is presented as a product $\mathbb{Z}_{m}\times \mathbb{Z}_{n}$,
and consider $W=\bigoplus_{a,b\in G}W_{a,b}$ a $(1,2)$-symmetric $G$-graded
twisted $\mathbb{C}$-algebra, with a fixed basis $\{x_{a,b}~:~x_{a,b}\in
W_{a,b}\}$ and structure constant $\widetilde{C}:G\times G\rightarrow
A\subset \mathbb{C}^{\ast }$, $A$ a finite subgroup of $\mathbb{C}^{\ast }$.
When $G=\mathbb{Z}_{n}$ is a cyclic group with generator $g\in G$, a basis
for $W$ was called \textit{standard} if it had the form $%
\{1,w_{g}^{(1)},w_{g}^{(2)},\dots ,w_{g}^{(n-1)}\}$ where $%
w_{g}^{(i)}=w_{g}\cdot w_{g}^{(i-1)}$ and $w_{g}\cdot w_{g}^{(n-1)}=1$ \cite%
{jwn}. We generalize this construction as follows. We may think of the
graduation of $W$ as an array of the following form: 
\begin{equation*}
\xymatrix{W_{0,0} & W_{0,1} & W_{0,2} & \cdots & W_{0,n-1} \\ W_{1,0} &
W_{1,1} & W_{1,2} & \cdots & W_{1,n-1} \\ \vdots & \vdots & \vdots & \cdots
& \vdots \\ W_{m-1,0} & W_{m-1,1} & W_{m-1,2} & \cdots & W_{m-1,n-1}}.
\end{equation*}%
As in the cyclic case, we choose standard bases for the first row and first
column. These two bases will be denoted by $\{1,w_{0,1},w_{0,2},\dots
,w_{0,n-1}\}$ and $\{1,w_{1,0},w_{2,0},\dots ,w_{m-1,0}\}$, respectively.
Now for the i-th row define $w_{i,j}=w_{0,1}\cdot w_{i,j-1}$ for $j=1,\dots
,n-1$. We notice that $w_{0,1}\cdot w_{i,n-1}\in W_{i,0}$. Hence, $%
w_{0,1}\cdot w_{i,n-1}=\alpha _{i}\cdot w_{i,0},$ for some $\alpha _{i}\in A$%
. We call ~$\mathcal{B}=\{w_{i,j}~:~i=0,\dots ,m-1,~j=0,\dots ,n-1\}$~a 
\textit{standard basis} for~$W$.

Define $T_{i,j}:W\rightarrow W$ to be the linear transformation given by $%
T_{i,j}(x)=w_{i,j}\cdot x$. For $T_{0,1}$, its rational form consists of $m$%
-blocks where each one looks like 
\begin{equation*}
\left[ 
\begin{array}{cccccc}
0 & 0 & 0 & \cdots & 0 & \alpha _{i} \\ 
1 & 0 & 0 & \cdots & 0 & 0 \\ 
0 & 1 & 0 & \cdots & 0 & 0 \\ 
\vdots & \vdots & \vdots &  & \vdots & \vdots \\ 
0 & 0 & 0 & \cdots & 1 & 0%
\end{array}%
\right] ,
\end{equation*}%
where $\alpha _{0}=1,\alpha _{1},\dots ,\alpha _{m-1}$ are elements in~$A$%
~such that~$T_{0,1}(w_{i,n-1})=\alpha _{i}\cdot w_{i,0}$. Let us denote by $%
\{e_{i,j}\}_{j=0,\dots ,n-1}$ the $n$-th roots of $\alpha _{i}$,~for $%
i=0,\dots ,m-1$. A straightforward computation shows that $e_{i,j}$~is an
eigenvalue of $T_{0,1}$ with eigenvector 
\begin{equation}
z_{i,j}=\sum_{k=0}^{n-1}e_{i,j}^{-k}\cdot w_{i,k}.  \label{e29}
\end{equation}%
From now on, the elements of the group $G$ will be denoted by products of
the form $a^{r}b^{s}$, with $0\leq r\leq m-1$, $0\leq s\leq n-1$. Also we
will write $w_{r,s}\cdot w_{i,j}=C(a^{r}b^{s},a^{i}b^{j})\cdot w_{r+i,s+j}$.
Notice that $T_{r,s}\circ T_{i,j}(x)=q(a^{r}b^{s},a^{i}b^{j}))\cdot
T_{i,j}\circ T_{r,s}(x)~$for every$~x\in W.$ Here, $%
q(a^{r}b^{s},a^{i}b^{j})=C(a^{r}b^{s},a^{i}b^{j})C(a^{i}b^{j},a^{r}b^{s})^{-1} 
$ (see \cite{jwn}). Therefore, 
\begin{eqnarray*}
T_{0,1}(T_{r,s}(z_{i,j}))
&=&q(b,a^{r}b^{s})T_{r,s}(T_{0,1}(z_{i,j}))~=~q(b,a^{r}b^{s})T_{r,s}(e_{i,j}z_{i,j})
\\
&=&q(b,a^{r}b^{s})e_{i,j}T_{r,s}(z_{i,j}).
\end{eqnarray*}%
Hence, $T_{r,s}(z_{i,j})$ is an eigenvector of $T_{0,1}$ associated to the
eigenvalue $q(b,a^{r}b^{s})e_{i,j}$. Since 
\begin{equation}
T_{r,s}(z_{i,j})=w_{r,s}\cdot
\sum_{k=0}^{n-1}e_{i,j}^{-k}w_{i,k}=%
\sum_{k=0}^{n-1}e_{i,j}^{-k}C(a^{r}b^{s},a^{i}b^{k}))w_{[r+i],[s+k]},
\label{e30}
\end{equation}%
where $[~]$ denotes the equivalence class in $\mathbb{Z}_{n}$ or $\mathbb{Z}%
_{m}$, we deduce that 
\begin{equation}
T_{r,s}(z_{i,j})=\eta _{r,s}^{i,j}\cdot z_{[r+i],l},  \label{eii}
\end{equation}%
for some $l \in \{0,1,\dots ,n-1\}$ and some $\eta _{r,s}^{i,j}\in K^{\ast }$%
. Also, since $q(b,a^{r}b^{s})e_{i,j}$ is an eigenvalue associated to $%
T_{r,s}(z_{i,j})$ we see that 
\begin{equation}
q(b,a^{r}b^{s})e_{i,j}=e_{[r+i],l}.  \label{e31}
\end{equation}%
By definition $e_{i,j}^{n}=\alpha _{i}$. Therefore, 
\begin{equation}
q(b,a^{r}b^{s})^{n}\alpha _{i}=\alpha _{\lbrack r+i]}.  \label{ei}
\end{equation}%
It follows from the definition of a standard basis that $C(b,a^{r}b^{s})=1,~%
\text{if}~s\neq n-1$, and that $C(b,a^{r}b^{n-1})=\alpha _{r}$. Therefore,
from equation (\ref{ei}) we obtain $\alpha
_{r}=C(a^{r}b^{s},b)^{-n},~r=1,2,\dots ,m$, when $s\neq n-1$, and $\alpha
_{r}^{n-1}=C(a^{r}b^{n-1},b)^{n},~r=1,2,\dots ,m$. Hence, as $%
q(b,a^{r}b^{s})=C(a^{r}b^{s},b)^{-1},$ when $s\neq n-1$, by replacing $%
q(b,a^{r}b^{s})$ in equation (\ref{ei}), we get $\alpha _{r}\alpha
_{i}=\alpha _{\lbrack r+i]}$. From this we see that $\alpha _{i}=\alpha
_{1}^{i}$. Finally, notice that $\alpha _{i}=C(a,b)^{-in}$. Since $\alpha
_{1}^{m}=1$, then $C(a,b)^{mn}=1$.

We summarize below what we have obtained so far: 
\begin{equation}
\begin{aligned} \alpha_{i} & =\alpha_{1}^{i},\\ \alpha_{i} & =
C(a,b)^{-in},\\ C(a^{r}b^{s},b)^{-n}&
=\alpha_{r}=(C(a,b)^{-r})^{n},~~\textrm{for}~ s \neq n-1,\\
C(a^{r}b^{n-1},b)^{n}& = \alpha_{r}^{n-1}=(C(a,b)^{-r(n-1)})^{n},\\
C(b^{s},b)^{n}& =1,~~~ \textrm{(cyclic case)}\\ C(a^{r},a)^{m}& =1,~~~
\textrm{(cyclic case)}\\ C(a,b)^{mn} & =1. \end{aligned}
\end{equation}%
On the other hand, by equation (\ref{e31}) it follows that 
\begin{equation*}
\eta _{r,s}^{i,j}\cdot Z_{[r+i],l}=\sum_{k=0}^{n-1}\eta _{r,s}^{i,j}\cdot
e_{[r+i],l}^{-k}\cdot w_{[r+i],k}=\sum_{k=0}^{n-1}\eta _{r,s}^{i,j}\cdot
q(b,a^{r}b^{s})^{-k}\cdot e_{i,j}^{-k}\cdot w_{[r+i],k}.
\end{equation*}%
Equations (\ref{e30}) and (\ref{eii}) imply the following identity: 
\begin{equation*}
\sum_{k=0}^{n-1}e_{i,j}^{-k}C(a^{r}b^{s},a^{i}b^{k})w_{[r+i],[s+k]}=%
\sum_{k=0}^{n-1}\eta _{r,s}^{i,j}\cdot q(b,a^{r}b^{s})^{-k}\cdot
e_{i,j}^{-k}\cdot w_{[r+i],k}.
\end{equation*}%
But the last equation is equivalent to 
\begin{equation*}
\sum_{k=s}^{n-1}e_{i,j}^{-(k-s)}C(a^{r}b^{s},a^{i}b^{k-s})w_{[r+i],k}=%
\sum_{k=s}^{n-1}\eta _{r,s}^{i,j}\cdot q(b,a^{r}b^{s})^{-k}\cdot
e_{i,j}^{-k}\cdot w_{[r+i],k},
\end{equation*}%
and also equivalent to 
\begin{equation*}
\sum_{k=0}^{s-1}e_{i,j}^{-(n+k-s)}C(a^{r}b^{s},a^{i}b^{n+k-s})w_{[r+i],k}=%
\sum_{k=0}^{s-1}\eta _{r,s}^{i,j}\cdot q(b,a^{r}b^{s})^{-k}\cdot
e_{i,j}^{-k}\cdot w_{[r+i],k}~~\text{for}~~s\neq 0.
\end{equation*}%
Therefore, the equations 
\begin{equation}
e_{i,j}^{-(k-s)}\cdot C(a^{r}b^{s},a^{i}b^{k-s})=\eta _{r,s}^{i,j}\cdot
q(b,a^{r}b^{s})^{-k}\cdot e_{i,j}^{-k}~\text{for}~k=s,s+1,\dots ,n-1,
\label{e32}
\end{equation}%
and 
\begin{equation}
e_{i,j}^{-(n+k-s)}\cdot C(a^{r}b^{s},a^{i}b^{n+k-s})=\eta _{r,s}^{i,j}\cdot
q(b,a^{r}b^{s})^{-k}\cdot e_{i,j}^{-k}~\text{for}~k=0,1,\dots ,s-1,~~\text{%
for}~s\neq 0  \label{e33}
\end{equation}%
hold.

From (\ref{e32}) if we let $k=s$, we deduce that $\eta
_{r,s}^{i,j}=C(a^{r}b^{s},a^{i})\cdot q(b,a^{r}b^{s})^{s}\cdot e_{i,j}^{s}.$
Replacing the last equation in (\ref{e32}) and in (\ref{e33}) we get that: 
\begin{equation}
C(a^{r}b^{s},a^{i}b^{l})=C(a^{r}b^{s},a^{i})\cdot C(b,a^{r}b^{s})^{-l}\cdot
C(a^{r}b^{s},b)^{l},~~\text{if}~~0\leq l<n-s,  \label{e34}
\end{equation}%
and 
\begin{equation}
C(a^{r}b^{s},a^{i}b^{l})=C(a^{r}b^{s},a^{i})\cdot C(b,a^{r}b^{s})^{n-l}\cdot
C(a^{r}b^{s},b)^{l-n}\cdot \alpha _{i},~~\text{if}~~n-s\leq l\leq
n-1;~~s\neq 0.  \label{e35}
\end{equation}%
Now we take in account the symmetry condition of $r$: $r(a,b,c)=r(b,a,c)$
for every $a,b,c\in G$. For any three general elements $%
a^{r}b^{s},a^{i}b^{k},a^{j}b^{l}\in G$ the symmetry condition looks like: 
\begin{equation}
C(a^{i}b^{k},a^{j}b^{l})C(a^{r}b^{s},a^{i+j}b^{k+l})C(a^{r}b^{s},a^{i}b^{k})^{-1}=C(a^{r}b^{s},a^{j}b^{l})C(a^{i}b^{k},a^{r+j}b^{s+l})C(a^{i}b^{k},a^{r}b^{s})^{-1}.
\label{e36}
\end{equation}%
Taking $k=0$ and $l=0$ in the above equation, and using the equation (\ref%
{e34}) and the fact that $C(a^{i},a^{j})=C(a^{i},a)^{j}$ (cyclic case, see 
\cite{jwn}) we obtain: $%
C(a^{r}b^{s},a^{i+j})=C(a^{r}b^{s},a^{i})C(a^{r}b^{s},a^{j})$. Then,
recursively, we get $C(a^{r}b^{s},a^{j})=C(a^{r}b^{s},a)^{j}$, and therefore 
$C(a^{r}b^{s},a)^{m}=1$. Also, notice that since $C(a^{r}b^{s},b)^{-n}=%
\alpha _{r}$ when $s\neq n-1$ and $C(a^{r}b^{n-1},b)^{-n}=\alpha _{r}^{1-n}$%
, we can rewrite equations (\ref{e34}) and (\ref{e35}) in the following
manner: 
\begin{equation*}
\begin{aligned}C(a^{r}b^{s},a^{i}b^{l})& =
C(a^{r}b^{s},a)^{i}C(a^{r}b^{s},b)^{l},~~ \textrm{for}~~ 0 \leq l \leq
n-s-1,\\ C(a^{r}b^{s},a^{i}b^{l})&
=C(a^{r}b^{s},a)^{i}C(a^{r}b^{s},b)^{l}\alpha_{r}\alpha_{i},~~
\textrm{for}~~ n-s \leq l \leq n-1~~ \textrm{and}~~ s \neq n-1,\\
C(a^{r}b^{n-1},a^{i}b^{l})&
=C(a^{r}b^{s},a)^{i}C(a^{r}b^{s},b)^{l}\alpha_{r}^{1-l}\alpha_{i}.
\end{aligned}
\end{equation*}%
We conclude that for a $(1,2)$-symmetric $\mathbb{Z}_{m}\times \mathbb{Z}%
_{n} $-graded twisted $\mathbb{C}$-algebra the structure constant $C:G\times
G\rightarrow A$ referred to a standard basis $\mathcal{B}$ must satisfy the
following equations: 
\begin{equation}
\begin{aligned} C(a^{r}b^{s},a^{i}b^{l})& =
C(a^{r}b^{s},a)^{i}C(a^{r}b^{s},b)^{l},~~ \textrm{for}~~ 0 \leq l \leq
n-s-1,\\ C(a^{r}b^{s},a^{i}b^{l})&
=C(a^{r}b^{s},a)^{i}C(a^{r}b^{s},b)^{l}\alpha_{r}\alpha_{i},~~
\textrm{for}~~ n-s \leq l \leq n-1~~ \textrm{and}~~ s \neq n-1,~ s \neq 0,\\
C(a^{r}b^{n-1},a^{i}b^{l})&
=C(a^{r}b^{n-1},a)^{i}C(a^{r}b^{n-1},b)^{l}\alpha_{r}^{1-l}\alpha_{i}~~
\textrm{if}~~ 1 \leq l \leq n-1,\\ C(a^{r}b^{s},a)^{m} & =1,\\ \alpha_{i} &
=\alpha_{1}^{i},\\ \alpha_{i} & = C(a,b)^{-in},\\ C(a^{r}b^{s},b)^{-n}&
=\alpha_{r}=(C(a,b)^{-r})^{n},~~\textrm{for}~ s \neq n-1,\\
C(a^{r}b^{n-1},b)^{n}& = \alpha_{r}^{n-1}=(C(a,b)^{-r(n-1)})^{n},\\
C(b^{s},b)^{n}& =1,~~~ \textrm{(cyclic case)}\\ C(a^{r},a)^{m}& =1,~~~
\textrm{(cyclic case)}\\ C(a,b)^{mn} & =1. \end{aligned}  \label{e50}
\end{equation}

Now we prove that two $(1,2)$-symmetric $G$-graded twisted algebras $W_{1}$
and $W_{2}$ are graded-isomorphic if and only if their structure constants
referred to standard bases are the same.

\begin{theorem}
\label{tii} Let $W_{1}$ and $W_{2}$ be $(1,2)$-symmetric $G$-graded twisted $%
\mathbb{C}$-algebras with standard bases $\mathcal{B}_{1}$ and $\mathcal{B}%
_{2}$, respectively, and associated structure constants $C_{1},C_{2}:G\times
G\rightarrow A$. Then, $W_{1}$ is graded-isomorphic to $W_{2}$ if and only
if $C_{1}=C_{2}$.
\end{theorem}

\begin{proof}
Suppose that~ $W_{1}$~ is graded-isomorphic to~ $W_{2}$.~ By Theorem \ref{ti}%
,~ $r_{1}=r_{2}$~ and~ $[C_{1}]=[C_{2}]$~ in~ $H^{2}(G,A)$.~ If~ $%
[C_{1}]=[C_{2}]$,~ then there exists~ $\rho :G\rightarrow A$~ such that $%
C_{1}=\partial ^{1}(\rho )C_{2}.$ That is:
\begin{equation}
C_{1}(a^{r}b^{s},a^{i}b^{k})=\rho (a^{i}b^{k})\rho (a^{r+i}b^{s+k})^{-1}\rho
(a^{r}b^{s})C_{2}(a^{r}b^{s},a^{i}b^{k}).  \label{ro}
\end{equation}%
Since the structure constants~ $C_{1},C_{2}$~ are referred to standard
bases, the following equalities hold for~ $j=1,2$:
\begin{eqnarray*}
C_{j}(1,a^{i}b^{k}) &=&C_{j}(a^{i}b^{k},1)=1,\text{ for all }i,k. \\
C_{j}(a,a^{i}) &=&1,~~\text{for}~~i=0,1,\dots ,m-1. \\
C_{j}(b,a^{i}b^{k}) &=&1,~~\text{for}~~i=0,1,\dots ,m-1,~~k=0,1,\dots ,n-2.
\\
C_{j}(b,a^{i}b^{n-1}) &=&\alpha _{i,(j)}.
\end{eqnarray*}%
These identities together with equation (\ref{ro}) yield: $\rho (a^{i})=\rho
(a)^{i}~~$for all$~~i=0,1,2,\dots ,n-1.$ In particular,~ $\rho (1)=1$~ and~ $%
\rho (a)^{m}=1$,~ and
\begin{equation*}
\rho (a^{i}b^{k})=\rho (a)^{i}\rho (b)^{k}~~\text{for}~~k\neq n.
\end{equation*}%
Moreover, $\rho (b^{i})=\rho (b)^{i}~~$for$~~i=0,1,\dots ,n$, since~
\begin{eqnarray*}
1 &=&C_{1}(b,b^{i})=\rho (b^{i})\rho (b^{i+1})^{-1}\rho (b)C_{2}(b,b^{i}) \\
&=&\rho (b^{i})\rho (b^{i+1})^{-1}\rho (b).
\end{eqnarray*}%
Therefore, $\rho (b)^{n}=\rho (b^{n})=1.$ All this can be summarize by
saying that~ $\rho :G\rightarrow A$~ is a group homomorphism. It immediately
follows ~ $\partial ^{1}(\rho )\equiv 1$~ what implies that~ $C_{1}=C_{2}$.
The reciprocal is clear.
\end{proof}

Finally, we want to see that if $C:G\times G\rightarrow A$ is a function
satisfying the identities stated in (\ref{e50}) then the vector space $%
\mathbb{C}^{m}\times \mathbb{C}^{n}$, endowed with the structure of a $G$%
-graded twisted algebra defined by the functions $C$ (referred to the
canonical basis of $\mathbb{C}^{m}\times \mathbb{C}^{n}$) is a $(1,2)$%
-symmetric $G$-graded twisted $\mathbb{C}$-algebra.

\begin{theorem}
\label{tiii}Let $G=\mathbb{Z}_{m}\times \mathbb{Z}_{n}$ and let $A\subset 
\mathbb{C}^{\ast }$ be a finite subgroup. Suppose that we choose values in $%
A $ for $C(a^{r}b^{s},a)$ and $C(a^{r}b^{s},b)$ satisfying the identities in
(\ref{e50}). Then $W=\mathbb{C}^{m}\times \mathbb{C}^{n}$ with the
multiplication given by $C$ (referred to the canonical basis of $\mathbb{C}%
^{m}\times \mathbb{C}^{n}$) is a $(1,2)$-symmetric $G$-graded twisted $%
\mathbb{C}$- algebra.
\end{theorem}

\begin{proof}
For~ $0\leq r,i\leq m-1,$~ and~ $0\leq s,l\leq n-1,$~ we define
\begin{equation*}
f(r,s,i,l)=\left\{
\begin{array}{cc}
1 & \text{if}~~0\leq l\leq n-s-1 \\
\alpha _{r}\alpha _{i} & ~~~~~\text{if}~~n-s\leq l\leq n-1,~~\text{and}%
~s\neq n-1 \\
\alpha _{r}^{1-l}\alpha _{i} & \text{if}~~s=n-1,~1\leq l\leq n-1%
\end{array}%
\right.
\end{equation*}%
From equation (\ref{e50})~ we deduce the identity
\begin{equation*}
C(a^{r}b^{s},a^{i}b^{l})=C(a^{r}b^{s},a)^{i}C(a^{r}b^{s},b)^{l}f(r,s,i,l).
\end{equation*}%
But we know that~ $%
r(a^{r}b^{s},a^{i}b^{k},a^{j}b^{l})=r(a^{i}b^{k},a^{r}b^{s},a^{j}b^{l})$~ if
and only if
\begin{equation*}
C(a^{i}b^{k},a^{j}b^{l})C(a^{r}b^{s},a^{i+j}b^{k+l})C(a^{r}b^{s},a^{i}b^{k})^{-1}=C(a^{r}b^{s},a^{j}b^{l})C(a^{i}b^{k},a^{r+j}b^{s+l})C(a^{i}b^{k},a^{r}b^{s})^{-1}.
\end{equation*}%
Therefore, this equation holds if and only if
\begin{equation}
\begin{aligned}&
f(i,k,j,l)f(r,s,i+j,[k+l])f(r,s,i,k)^{-1}C(a^{r}b^{s},b)^{[k+l]-(k+l)}=\\
&=f(r,s,j,l)f(i,k,r+j,[s+l])f(i,k,r,s)^{-1}C(a^{i}b^{k},b)^{[s+l]-(s+l)}%
\end{aligned},  \label{e52}
\end{equation}%
where $[~]$~ denotes residue classes in~ $\mathbb{Z}_{n}$. It is not
difficult to see that equation (\ref{e52}) holds by directly computing from
the identities in (\ref{e50}). For this, each one of the following cases
should be considered separately:
\begin{equation*}
\lbrack k+l]+s<n~~~\text{and}~~~[s+l]+k<n:~~\left\{
\begin{array}{c}
k+l\geq n,~s+l\geq n \\
k+l<n,~s+l<n \\
\end{array}%
\right.
\end{equation*}%
\begin{equation*}
\lbrack k+l]+s<n~~~\text{and}~~~[s+l]+k\geq n:~~\left\{
\begin{array}{c}
k+l\geq n,~s+l<n%
\end{array}%
\right.
\end{equation*}%
\begin{equation*}
\lbrack k+l]+s\geq n~~~\text{and}~~~[s+l]+k<n:~~\left\{
\begin{array}{c}
s+l\geq n,~k+l<n%
\end{array}%
\right.
\end{equation*}%
\begin{equation*}
\lbrack k+l]+s\geq n~~~\text{and}~~~[s+l]+k\geq n:~~\left\{
\begin{array}{c}
k+l\geq n,~s+l\geq n \\
k+l<n,~s+l<n%
\end{array}%
\right.
\end{equation*}
\end{proof}

We are ready to state the main theorem of this section:

\begin{theorem}
The number of (graded) isomorphism classes of $(1,2)$-symmetric $\mathbb{Z}%
_{m}\times \mathbb{Z}_{n}$-graded twisted $\mathbb{C}$-algebras with
structure constants taking values in a finite subgroup $A\subset \mathbb{C}%
^{\ast }$ is given by: 
\begin{equation*}
|R_{m}|^{mn-3}|R_{n}|^{mn-3}|R_{mn}|,
\end{equation*}%
where $R_{k}$ denotes the set of $k$-th roots of unity: $\{\omega \in
A~~:~~\omega ^{k}=1\}$.
\end{theorem}

\begin{proof}
From the discussion above, a~ $(1,2)$-symmetric~ $\mathbb{Z}_{m}\times
\mathbb{Z}_{n}$-graded twisted~ $\mathbb{C}$-algebra is determined, up to
graded isomorphisms, by the structure constant that is defined with respect
to a standard basis. In turn, this function is completely determined by all
possible choices of~ $C(a^{r}b^{s},a)$~ and~ $C(a^{r}b^{s},b)$, satisfying
the identities in~ (\ref{e50}). As~ $C(a^{r}b^{s},a)^{m}=1$~ for all~ $0\leq
r\leq m-1,~0\leq s\leq n-1,$~ and~
\begin{equation*}
1=C(1,a)=C(a,a)=C(b,a),
\end{equation*}%
~ then we see that there are~ $|R_{m}|^{mn-3}$~ possible choices for~ $%
C(a^{r}b^{s},a)$. Similarly, as~ $C(b^{s},b)^{n}=1$~ for~ $0\leq s\leq n-1$%
,~ and~ $C(1,b)=1=C(b,b)$,~ then~ $C(b^{s},b)$~ may be chosen in~ $%
|R_{n}|^{n-2}$~ possible ways. Since~ $C(a,b)^{mn}=1$,~ there are~ $|R_{mn}|$%
~ possible values for~ $C(a,b)$. In the case where~ $s\neq n-1$~ and~ $r\neq
0$,~ the identities in (\ref{e50}) tell us that~
\begin{equation*}
C(a^{r}b^{s},b)^{n}=C(a,b)^{rn}=(C(a,b)^{r})^{n}.
\end{equation*}%
~ Therefore,~ $C(a^{r}b^{s},b)=\omega C(a,b)^{r}$,~ where~ $\omega $~ is
some fixed~ $n$-th root of unity.~ Thus, there are~ $|R_{n}|^{(n-1)(m-1)-1}$%
~ possible choices for~ $C(a^{r}b^{s},b)$,~ if~ $s\neq n-1$~ and~ $r\neq 0$.

Finally, again by using (\ref{e50}) we obtain:~ $%
C(a^{r}b^{n-1},b)^{n}=(C(a,b)^{-r(n-1)})^{n}$,~ and therefore~ $%
C(a^{r}b^{n-1},b)=\omega C(a,b)^{-r(n-1)}$,~ where~ $\omega ^{n}=1$.~ Hence,
if~ $r\geq 1$,~ the value of~ $C(a^{r}b^{n-1},b)$~ can be chosen in~ $%
|R_{n}|^{m-1}$~ possible manners. In conclusion, the number of algebras
satisfying the hypothesis of the theorem is given by
\begin{equation*}
|R_{m}|^{mn-3}|R_{n}|^{n-2+(n-1)(m-1)-1+m-1}|R_{mn}|=|R_{m}|^{mn-3}|R_{n}|^{mn-3}|R_{mn}|.
\end{equation*}
\end{proof}

\begin{remark}
If $m$ and $n$ are relatively prime, then $\mathbb{Z}_{m}\times \mathbb{Z}%
_{n}\cong \mathbb{Z}_{mn}$, and since $|R_{mn}|=|R_{m}||R_{n}|$, the above
number is equal to $|R_{mn}|^{mn-2}$ which gives the correct number of
non-isomorphic algebras in the cyclic case, as provided in \cite{jwn}.
\end{remark}

Now, we discuss the general case of any finite abelian group, presented as $%
G=\mathbb{Z}_{n_{1}}\times \cdots \times \mathbb{Z}_{n_{k}}$. Since the
constructions are almost the same as the case of the product of two cyclic
groups, in the rest of this article we will limit ourselves to give a sketch
of the main arguments.

First, we start by defining a generalized standard basis for a $G$-graded
twisted algebra $W$. We proceed by induction on the number of factors.
Suppose $G=G_{1}\times \mathbb{Z}_{n_{k}}$, where $G_{1}=\mathbb{Z}%
_{n_{1}}\times \cdots \times \mathbb{Z}_{n_{k-1}}$. Fix $a_{i}\in \mathbb{Z}%
_{n_{i}}$ a generator. We define a standard basis for $W$ as a basis of the
form $\{w_{g,j}\}$, where $w_{e,j}=w_{a_{k}}^{(j)}$ ($e$ denotes the
identity element of $G_{1}$), as it was defined in the cyclic case, and
where $\{w_{g,0}=w_{g}~:~g\in G_{1}\}$ is a standard basis for $W$
restricted to $G_{1}$, and $w_{g,j}=w_{a_{k}}^{(1)}\cdot w_{g,j-1}$, for $%
g\neq e$ and $j\neq 0$. So for each $g\in G_{1}$ there is $\alpha _{g}\in A$
such that $w_{a_{k}}^{(1)}\cdot w_{g,n_{k}-1}=\alpha _{g}\cdot w_{g}$.

\begin{remark}
Notice that since $\mathcal{B}$ was defined recursively, its restriction to $%
H=\mathbb{Z}_{n_{1}}\times \cdots \times \mathbb{Z}_{n_{t}}$ with $1\leq
t\leq k$, is a standard basis for $W|_{H}$.
\end{remark}

Let $W$ be a $(1,2)$-symmetric $G$-graded twisted $\mathbb{C}$-algebra, with
standard basis $\mathcal{B}$ and structure constant $C:G\times G\rightarrow
A $. Consider as before $T_{a_{k}}:W\rightarrow W$ the linear transformation
given by $T_{a_{k}}(x)=w_{a_{k}}\cdot x$. For each $g\in G_{1}$, let $%
\{e_{g,j}\}$ denotes the set of $n_{k}$-th roots of $\alpha _{g}$. Notice
that $\alpha _{e}=1$ and therefore $\{e_{e,j}\}$ is the set of $n_{k}$-roots
of unity. Define 
\begin{equation}
z_{g,j}=\sum_{\mu =0}^{n_{k}-1}e_{g,j}^{-\mu }w_{g,\mu }.  \label{e56}
\end{equation}%
Then, as before, $z_{g,j}$ is an eigenvector of $T_{a_{k}}$ associated to
the eigenvalue $e_{g,j}$. Since $W$ is $(1,2)$-symmetric we have 
\begin{equation*}
T_{a_{k}}(T_{g^{\prime },s}(z_{g,j}))=q(a_{k},g^{\prime }\cdot
a_{k}^{s})\cdot e_{g,j}\cdot T_{g^{\prime },s}(z_{g,j}),
\end{equation*}%
(see \cite{jwn}). Therefore $T_{g^{\prime },s}(z_{g,j})$ is an eigenvector
of $T_{a_{k}}$ associated to the eigenvalue $q(a_{k},g^{\prime }\cdot
a_{k}^{s})\cdot e_{g,j}$. Here, $T_{g^{\prime },s}$ denotes the linear
transformation given by $T_{g^{\prime },s}(x)=w_{g^{\prime },s}\cdot x$.
Since 
\begin{equation}
T_{g^{\prime },s}(z_{g,j})=\sum_{\mu =0}^{n_{k}-1}e_{g,j}^{-\mu }\cdot
C(g^{\prime }\cdot a_{k}^{s},g\cdot a_{k}^{\mu })\cdot w_{g^{\prime
}g,[s+\mu ]}  \label{e57}
\end{equation}%
where $[~]$ denotes residue classes in $\mathbb{Z}_{n_{k}}$, then it holds
that 
\begin{equation}
T_{g^{\prime },s}(z_{g,j})=\eta _{g^{\prime },s}^{g,j}\cdot z_{g^{\prime
}g,l},~~\text{for some}~~0\leq l\leq n_{k}-1.  \label{e58}
\end{equation}%
Hence, 
\begin{equation}
q(a_{k},g^{\prime }\cdot a_{k}^{s})\cdot e_{g,j}=e_{g^{\prime }g,l}.
\label{e59}
\end{equation}%
This last equation is the generalization of equation (\ref{e31}). As in that
case, we may derive the following identities: If $%
g=a_{1}^{r_{1}}a_{2}^{r_{2}}\dots a_{k-1}^{r_{k-1}}$, then 
\begin{equation}
\begin{aligned} \alpha_{g} & = C(a_{1},a_{k})^{-r_{1}n_{k}} \cdots
C(a_{k-1},a_{k})^{-r_{k-1}n_{k}}\\ \alpha_{a_{i}} &
=C(a_{i},a_{k})^{-n_{k}}\\ C(a_{k},g \cdot a_{k}^{n_{k}-1}) & = \alpha_{g} =
C(g,a_{k})^{-n_{k}}\\ C(g \cdot a_{k}^{s},a_{k}) & = \omega_{g,s} \cdot
C(a_{1},a_{k})^{r_{1}} \cdots C(a_{k-1},a_{k})^{r_{k-1}},~~ \textrm{for}~~ s
\neq n_{k}-1,~~ \textrm{where}~~ \omega_{g,s}^{n_{k}}=1\\ C(g \cdot
a_{k}^{n_{k}-1},a_{k}) & = \omega_{g} \cdot C(a_{1},a_{k})^{-r_{1}(n_{k}-1)}
\cdots C(a_{k-1},a_{k})^{-r_{k-1}(n_{k}-1)}~~ \textrm{where}~~
\omega_{g}^{n_{k}}=1\\ C(a_{i},a_{k})^{n_{i}n_{k}} & = 1. \end{aligned}
\label{eiii}
\end{equation}%
Now, equations (\ref{e57}), (\ref{e58}) and (\ref{e59}) imply: 
\begin{eqnarray*}
\sum_{\mu =0}^{n_{k}-1}e_{g,j}^{-\mu }\cdot C(g^{\prime }\cdot
a_{k}^{s},g\cdot a_{k}^{\mu })\cdot w_{g^{\prime }g,[s+\mu ]} &=&\eta
_{g^{\prime },s}^{g,j}\cdot z_{g^{\prime }g,l} \\
&=&\sum_{\mu =0}^{n_{k}-1}\eta _{g^{\prime },s}^{g,j}\cdot e_{g^{\prime
}g,l}^{-\mu }\cdot w_{g^{\prime }g,\mu } \\
&=&\sum_{\mu =0}^{n_{k}-1}\eta _{g^{\prime },s}^{g,j}\cdot q(a_{k},g^{\prime
}\cdot a_{k}^{s})^{-\mu }\cdot e_{g,j}^{-\mu }\cdot w_{g^{\prime }g,\mu }
\end{eqnarray*}%
Similarly, as in the case of a product of two cyclic groups, the last
equation implies that for every $g,g^{\prime }\in G_{1}$, 
\begin{equation}
\begin{aligned} C(g' \cdot a_{k}^{s},g \cdot a_{k}^{l}) & = C(g' \cdot
a_{k}^{s},g) \cdot C(g' \cdot a_{k}^{s},a_{k})^{l},~~ \textrm{if}~~ 0 \leq l
< n_{k}-s.\\ C(g' \cdot a_{k}^{s},g \cdot a_{k}^{l}) & = C(g' \cdot
a_{k}^{s},g) \cdot C(g' \cdot a_{k}^{s},a_{k})^{l-n_{k}} \cdot \alpha_{g},~~
\textrm{if}~~ n_{k}-s \leq l \leq n_{k}-1;~~ s \neq 0,~~ s \neq n_{k}-1.\\
C(g' \cdot a_{k}^{n_{k}-1},g \cdot a_{k}^{l}) & = C(g' \cdot
a_{k}^{n_{k}-1},g) \cdot C(g' \cdot a_{k}^{n_{k}-1},a_{k})^{l-n_{k}} \cdot
\alpha_{g'}^{n_{k}-l} \cdot \alpha_{g}. \end{aligned}  \label{e70}
\end{equation}%
That proves the following theorem:

\begin{theorem}
\label{t1} Suppose that $G=\mathbb{Z}_{n_{1}}\times \mathbb{Z}_{n_{2}}\times
\cdots \times \mathbb{Z}_{n_{k}}$, and fix generators $a_{1},a_{2},\dots
,a_{k}$, $a_{i}\in \mathbb{Z}_{n_{i}}$. Suppose that $W$ is a $(1,2)$%
-symmetric $G$-graded twisted $\mathbb{C}$-algebra. If $\mathcal{B}$ is a
standard basis for $W$ with structure constant $C:G\times G\rightarrow 
\mathbb{C}^{\ast }$, then:\newline

If$~~0\leq l<n_{k}-s:$~\newline

$C(g^{\prime }\cdot a_{k}^{s},g\cdot a_{k}^{l})=C(g^{\prime }\cdot
a_{k}^{s},g)\cdot C(g^{\prime }\cdot a_{k}^{s},a_{k})^{l},$~\newline

If$~~n_{k}-s\leq l\leq n_{k}-1,~~s\neq 0,~~s\neq n_{k}-1:$~\newline

$C(g^{\prime }\cdot a_{k}^{s},g\cdot a_{k}^{l})=C(g^{\prime }\cdot
a_{k}^{s},g)\cdot C(g^{\prime }\cdot a_{k}^{s},a_{k})^{l-n_{k}}\cdot \alpha
_{g}.$~\newline

For $s=n_{k}-1:$~\newline

$C(g^{\prime }\cdot a_{k}^{n_{k}-1},g\cdot a_{k}^{l})=C(g^{\prime }\cdot
a_{k}^{n_{k}-1},g)\cdot C(g^{\prime }\cdot
a_{k}^{n_{k}-1},a_{k})^{l-n_{k}}\cdot \alpha _{g^{\prime }}^{n_{k}-l}\cdot
\alpha _{g}.$~\newline

If the element$~g~$is written as$~g=a_{1}^{r_{1}}a_{2}^{r_{2}}\dots
a_{k-1}^{r_{k-1}}~~$then:~\newline

$\alpha _{g}=C(a_{1},a_{k})^{-r_{1}n_{k}}\cdots
C(a_{k-1},a_{k})^{-r_{k-1}n_{k}}$

$\alpha _{a_{i}}=C(a_{i},a_{k})^{-n_{k}},$

$C(a_{k},g\cdot a_{k}^{n_{k}-1})=\alpha _{g}=C(g,a_{k})^{-n_{k}}.$~\newline

For$~~s\neq n_{k}-1,~~$where$~~\omega _{g,s}^{n_{k}}=1:$~\newline

$C(g\cdot a_{k}^{s},a_{k})=\omega _{g,s}\cdot C(a_{1},a_{k})^{r_{1}}\cdots
C(a_{k-1},a_{k})^{r_{k-1}},$

$C(g\cdot a_{k}^{n_{k}-1},a_{k})=\omega _{g}\cdot
C(a_{1},a_{k})^{-r_{1}(n_{k}-1)}\cdots C(a_{k-1},a_{k})^{-r_{k-1}(n_{k}-1)}$

$C(a_{i},a_{k})^{n_{i}n_{k}}=1.$
\end{theorem}

As in the case of a product of two factors, the $(1,2)$-symmetry provides
some extra information about the structure constant $C$ that we summarize in
the following two lemmas. Their proofs follow the same lines as before and
will be omitted.

\begin{lemma}
\label{l2} Suppose that $G=\mathbb{Z}_{n_{1}}\times \mathbb{Z}_{n_{2}}\times
\cdots \times \mathbb{Z}_{n_{k}}$, and fix generators $a_{1},a_{2},\dots
,a_{k}$, $a_{i}\in \mathbb{Z}_{n_{i}}$. Suppose that $W$ is a $(1,2)$%
-symmetric $G$-graded twisted $\mathbb{C}$-algebra. If $\mathcal{B}$ is a
standard basis for $W$ with structure constant $C:G\times G\rightarrow 
\mathbb{C}^{\ast }$ then: 
\begin{equation*}
C(g_{1}\cdot a_{k}^{s},g_{2}g_{3})=C(g_{1}\cdot a_{k}^{s},g_{2})\cdot
C(g_{1}\cdot a_{k}^{s},g_{3})\cdot C(g_{2},g_{3})^{-1}\cdot
C(g_{2},g_{1}g_{3})\cdot C(g_{2},g_{1})^{-1}
\end{equation*}%
for every $g_{1},g_{2},g_{3}\in G_{1}=\mathbb{Z}_{n_{1}}\times \cdots \times 
\mathbb{Z}_{n_{k-1}}$.

Furthermore, if $g=a_{1}^{r_{1}}\cdots a_{i}^{r_{i}}\cdots
a_{k-1}^{r_{k-1}}, $ and if we denote $\widetilde{g}=a_{1}^{r_{1}}\cdots
a_{i-1}^{r_{i-1}}$, then 
\begin{equation*}
C(g\cdot a_{k}^{s},a_{i}^{j})=C(g\cdot a_{k}^{s},a_{i})^{j}\cdot C(a_{i},%
\widetilde{g}\cdot a_{i}^{r_{i}})^{-(j-1)}\cdot C(a_{i},\widetilde{g}\cdot
a_{i}^{r_{i}+1})\cdots C(a_{i},\widetilde{g}\cdot a_{i}^{r_{i}+j-1}).
\end{equation*}%
Hence, if $r_{i}=n_{i}-1:$ 
\begin{equation*}
C(g\cdot a_{k}^{s},a_{i}^{j})=C(g\cdot a_{k}^{s},a_{i})^{j}\cdot C(%
\widetilde{g},a_{i})^{(j-1)n_{i}}~~\text{if}~~r_{i}=n_{i}-1.
\end{equation*}%
On the other hand, if $r_{i}\neq n_{i}-1:$ 
\begin{equation*}
C(g\cdot a_{k}^{s},a_{i}^{j})=C(g\cdot a_{k}^{s},a_{i})^{j}\cdot C(%
\widetilde{g},a_{i})^{-n_{i}}~~\text{or}~~C(g\cdot
a_{k}^{s},a_{i}^{j})=C(g\cdot a_{k}^{s},a_{i})^{j}.
\end{equation*}%
In particular, 
\begin{equation*}
C(g\cdot a_{k}^{s},a_{i})^{n_{i}}=C(a_{i},\widetilde{g}\cdot
a_{i}^{r_{i}})^{n_{i}-1}\cdot C(a_{i},\widetilde{g}\cdot
a_{i}^{r_{i}+1})^{-1}\cdots C(a_{i},\widetilde{g}\cdot
a_{i}^{r_{i}+n_{i}-1})^{-1}.
\end{equation*}%
Therefore, 
\begin{equation*}
\begin{aligned} C(g \cdot a_{k}^{s},a_{i})^{n_{i}} & =
(C(\widetilde{g},a_{i})^{-(n_{i}-1)})^{n_{i}}~~ \textrm{if}~~
r_{i}=n_{i}-1,\\ C(g \cdot a_{k}^{s},a_{i})^{n_{i}} &
=C(\widetilde{g},a_{i})^{n_{i}}~~ \textrm{if}~~ r_{i} \neq n_{i}-1.
\end{aligned}
\end{equation*}
\end{lemma}

\begin{remark}
From the Theorem \ref{t1} we know that 
\begin{equation*}
C(\widetilde{g},a_{i})= \omega_{\widetilde{g}} \cdot C(a_{1},a_{i})^{r_{1}}
\cdots C(a_{i-1},a_{i})^{r_{i-1}},~~ \text{where}~~ \omega_{\widetilde{g}%
}^{n_{i}}=1.
\end{equation*}
Therefore, the last lemma implies that 
\begin{equation*}
C(g \cdot a_{k}^{s},a_{i})^{n_{i}}=(C(a_{1},a_{i})^{r_{1}} \cdot
C(a_{i-1},a_{i})^{r_{i-1}})^{-(n_{i}-1)n_{i}},~~ \text{if}~~ r_{i}=n_{i}-1,
\end{equation*}
and that 
\begin{equation*}
C(g \cdot a_{k}^{s},a_{i})^{n_{i}}=(C(a_{1},a_{i})^{r_{1}} \cdot
C(a_{i-1},a_{i})^{r_{i-1}})^{n_{i}},~~ \text{if}~~ r_{i} \neq n_{i}-1.
\end{equation*}
Hence, 
\begin{equation}  \label{e74}
\begin{aligned} C(g \cdot a_{k}^{s},a_{i}) & = \omega_{g,s}
C(a_{1},a_{i})^{r_{1}} \cdot C(a_{i-1},a_{i})^{r_{i-1}}~~ \textrm{if}~~
r_{i} \neq n_{i}-1,~~ \textrm{where}~~ \omega_{g,s}^{n_{i}}=1,\\ C(g \cdot
a_{k}^{s},a_{i}) & = \omega_{g,s} \cdot (C(a_{1},a_{i})^{r_{1}} \cdot
C(a_{i-1},a_{i})^{r_{i-1}})^{-(n_{i}-1)},~~ \textrm{if}~~ r_{i}= n_{i}-1,~~
\textrm{where}~~ \omega_{g,s}^{n_{i}}=1. \end{aligned}
\end{equation}
\end{remark}

\begin{lemma}
\label{l3} Let $G=\mathbb{Z}_{n_{1}}\times \cdots \times \mathbb{Z}_{n_{k}}$%
, and fix generators $a_{1},a_{2},\dots ,a_{k}$, $a_{i}\in \mathbb{Z}%
_{n_{i}} $. Suppose that $W$ is a $(1,2)$-symmetric $G$-graded twisted $%
\mathbb{C}$-algebra. Then the constants $C(g\cdot a_{k}^{s},g^{\prime })$
always can be expressed in terms of the constants $C(g\cdot a_{k}^{s},a_{i})$%
, $1\leq i\leq k-1$ and $C(a_{j},a_{t})$, $1\leq j,t\leq k-1$, for every $%
g,g^{\prime }\in G_{1}=\mathbb{Z}_{n_{1}}\times \cdots \times \mathbb{Z}%
_{n_{k-1}}$.
\end{lemma}

The following theorem summarizes the above discussion:

\begin{theorem}
\label{t2} Let $G$ be presented as $\mathbb{Z}_{n_{1}}\times \mathbb{Z}%
_{n_{2}}\times \cdots \times \mathbb{Z}_{n_{k}}$, and fix generators $%
a_{1},a_{2},\dots ,a_{k}$, $a_{i}\in \mathbb{Z}_{n_{i}}$. Suppose that $W$
is a $(1,2)$-symmetric $G$-graded twisted $\mathbb{C}$-algebra. If $\mathcal{%
B}$ is a standard basis for $W$ with structure constant $C:G\times
G\rightarrow \mathbb{C}^{\ast }$, then the values of $C(g\cdot
a_{k}^{s},a_{i}),C(g\cdot a_{k}^{s},a_{k})$ and $C(a_{t},a_{j})$, with $%
1\leq i,t,j\leq k-1$, $g\in G_{1}=\mathbb{Z}_{n_{1}}\times \cdots \times 
\mathbb{Z}_{n_{k-1}}$ completely determine the structure constant $C$.
Furthermore, the following identities generalize the ones obtained in (\ref%
{e50}):\newline

If$~~0\leq l<n_{k}-s,$ then ~\newline

$C(g\cdot a_{k}^{s},g^{\prime }\cdot a_{k}^{l})=C(g\cdot a_{k}^{s},g^{\prime
})\cdot C(g\cdot a_{k}^{s},a_{k})^{l}.$\newline

If$~~n_{k}-s\leq l\leq n_{k}-1,~~s\neq 0,~~s\neq n_{k}-1,$ then \newline

$C(g\cdot a_{k}^{s},g^{\prime }\cdot a_{k}^{l})=C(g\cdot a_{k}^{s},g^{\prime
})\cdot C(g\cdot a_{k}^{s},a_{k})^{l-n_{k}}\cdot \alpha _{g^{\prime }}.$%
\newline

If $s=n_{k}-1$ then\newline

$C(g\cdot a_{k}^{n_{k}-1},g^{\prime }\cdot a_{k}^{l})=C(g\cdot
a_{k}^{n_{k}-1},g^{\prime })\cdot C(g\cdot
a_{k}^{n_{k}-1},a_{k})^{l-n_{k}}\cdot \alpha _{g}^{n_{k}-l}\cdot \alpha
_{g^{\prime }}.$ \newline

If$~g~$is written as$~g=a_{1}^{r_{1}}a_{2}^{r_{2}}\dots a_{k-1}^{r_{k-1}}~~$%
then\newline

$\alpha _{g}=C(a_{1},a_{k})^{-r_{1}n_{k}}\cdots
C(a_{k-1},a_{k})^{-r_{k-1}n_{k}}$

$\alpha _{a_{i}}=C(a_{i},a_{k})^{-n_{k}}$

$C(a_{i},a_{k})^{n_{i}n_{k}}=1.$

$C(a_{k},g\cdot a_{k}^{n_{k}-1})=\alpha _{g}=C(g,a_{k})^{-n_{k}}$

$C(g\cdot a_{k}^{n_{k}-1},a_{k})=\omega _{g}\cdot
C(a_{1},a_{k})^{-r_{1}(n_{k}-1)}\cdots
C(a_{k-1},a_{k})^{-r_{k-1}(n_{k}-1)},~~$where$~~\omega _{g}^{n_{k}}=1.$%
\newline

For the case where$~~s\neq 0~$and$~g\neq a_{j},~$for$~j=1,2,\dots ,k:$%
\newline

$C(g\cdot a_{k}^{s},a_{k})=\omega _{g,s}\cdot C(a_{1},a_{k})^{r_{1}}\cdots
C(a_{k-1},a_{k})^{r_{k-1}},~~for~~s\neq n_{k}-1,~~where~~\omega
_{g,s}^{n_{k}}=1.$\newline

For the case$~~s\neq 0~$and$~g\neq a_{j}~$for$~j=1,2,\dots ,k-1$, and$%
~i=1,2,\dots ,k-1~$we have$:$\newline

$C(g\cdot a_{k}^{s},a_{i})=\omega _{g,s}C(a_{1},a_{i})^{r_{1}}\cdots
C(a_{i-1},a_{i})^{r_{i-1}}~~$if$~~r_{i}\neq n_{i}-1,~~$where$~~\omega
_{g,s}^{n_{i}}=1.$

$C(g\cdot a_{k}^{s},a_{i})=\omega _{g,s}\cdot (C(a_{1},a_{i})^{r_{1}}\cdots
C(a_{i-1},a_{i})^{r_{i-1}})^{-(n_{i}-1)},~~$if$~~r_{i}=n_{i}-1,~~$where$%
~~\omega _{g,s}^{n_{i}}=1.$\newline
\end{theorem}

Before starting our next theorem we notice that as in Theorem \ref{tii}
before two $(1,2)$-symmetric $\mathbb{Z}_{n_{1}}\times \cdots \times \mathbb{%
Z}_{n_{k}}$-graded twisted algebras are graded-isomorphic if and only if
they have the same structure constants in their respective standard bases.
The following theorem generalizes Theorem \ref{tiii} when $G$ is a product
of an arbitrary number of cyclic groups.

\begin{theorem}
Let $G=\mathbb{Z}_{n_{1}}\times \cdots \times \mathbb{Z}_{n_{k}}$ and let $%
A\subset \mathbb{C}^{\ast }$ be a finite subgroup. Suppose that we choose
values in $A$ for $C(g\cdot a_{k}^{s},a_{k})$, $C(a_{i},a_{j})$ and $%
C(g\cdot a_{k}^{s},a_{i})$ satisfying the identities in Theorem \ref{t2}.
Then $W=\mathbb{C}^{n_{1}}\times \cdots \times \mathbb{C}^{n_{k}}$ with the
multiplication given by $C$ (referred to the canonical basis of $\mathbb{C}%
^{n_{1}}\times \cdots \times \mathbb{C}^{n_{k}}$) is a $(1,2)$-symmetric $G$%
-graded twisted $\mathbb{C}$- algebra.
\end{theorem}

Finally, we may state our main theorem:

\begin{theorem}
\label{t3} The number of non-(graded) isomorphic $(1,2)$-symmetric $G=%
\mathbb{Z}_{n_{1}}\times \cdots \times \mathbb{Z}_{n_{k}}$-graded twisted $%
\mathbb{C}$-algebras with structure constants taking values in a finite
subgroup $A\subset \mathbb{C}^{\ast }$ is given by the product: 
\begin{equation*}
\prod_{i=1}^{k}|R_{n_{i}}|^{|G|-(k+1)}\prod_{1\leq i<j\leq
k}|R_{n_{i}n_{j}}|,
\end{equation*}%
where $R_{n_{i}}$ denotes the set $\{\omega \in A~~:~~\omega ^{n_{i}}=1\}$.
\end{theorem}

We have given a classification of a concrete family of twisted group
algebras satisfying a particular property in their function $r$. The diverse
properties that can be imposed to the function $r,$ and to the function $q$,
lead to further families of algebras which deserve as well classification
endeavors. In a forth coming publication we will tackle the relation of
these properties and more standard families of algebras such as the diverse
sorts of Lie-admissible algebras. In this sense, the approach of studying
twisted group algebras via the properties of their $r$ and $q$ functions,
provide novel powerful tools to classify concrete subclasses of standard
algebras such as left-symmetric, pre-Lie, anti-flexible, and further
varieties of algebras.

\end{document}